\documentclass{monthlyarticle}

\usepackage{psfrag}
\usepackage{mathrsfs}
\usepackage{amsmath}
\usepackage{url}
\usepackage{amssymb,amsthm}
\usepackage{graphicx,color}
\usepackage{epstopdf}
\usepackage[backgroundcolor=green!40, linecolor=green!40]{todonotes}

\def\({\left(}
\def\){\right)}
\def\le{\leqslant}
\def\ge{\geqslant}
\def\leq{\leqslant}
\def\geq{\geqslant}

\def\orbitset{\mathscr{O}}
\def\orbitnumber{\mathsf{O}}
\def\fixset{\mathscr{F}}
\def\fixnumber{\mathsf{F}}

\def\Z{\mathbb{Z}}
\def\ord{\operatorname{ord}}

\let\nvert\nmid

\newtheorem{theorem}{Theorem}
\newtheorem{lemma}[theorem]{Lemma}
\newtheorem{corollary}[theorem]{Corollary}
\newtheorem{proposition}[theorem]{Proposition}

\theoremstyle{definition}

\newtheorem{example}[theorem]{Example}

\newtheorem{remark}[theorem]{Remark}

\renewcommand\today{}

\begin{document}

\title{Halving dynamical systems}
\author{Shaun Stevens, Tom Ward,
and Stefanie Zegowitz}

\maketitle

%\begin{abstract}
%
%\end{abstract}

%%%%%%%%%%%%%%%%%%%%%%%%%%%%%%%%%%%%%%%%%%%%%%%%%%%%%%%%
\section{Introduction.}
%%%%%%%%%%%%%%%%%%%%%%%%%%%%%%%%%%%%%%%%%%%%%%%%%%%%%%%%

Halving (or doubling) a number is an
innocuous operation.
Unarguably it changes the size of a number, though
for a rational number this change in the real size
is exactly compensated by a change in its size
as a~$2$-adic number. For a real number all the more
interesting properties --- being rational, algebraic,
transcendent, well- or badly-approximable --- are unchanged.
Doubling (or halving) other objects also makes sense, but
with potentially much greater impact on non-trivial properties.

For example, doubling a finite group~$G$ may be
thought of
as forming another group~$H$ for which
the sequence of groups and homomorphisms
\[
\{1\}\longrightarrow
C_2
\longrightarrow
H
\longrightarrow
G
\longrightarrow\{1\}
\]
is exact, where~$C_2$ is the group with two elements.
Once again this has a simple effect on the size, as~$\vert H\vert=2\vert G\vert$.
However, the structure of the group clearly can change,
and the property of being abelian or simple
is not generally preserved.
Halving only makes sense if~$\vert G\vert$ is even, in
which case it might be thought of as follows.
If~$\vert G\vert$ is even then~$G$ contains at least
one copy of~$C_2$, and halving means forming the
quotient space~$G/C_2$. This operation does not
even preserve the property of being a group,
unless the copy of~$C_2$ happens to be a normal
subgroup.

Our purpose here is to discuss doubling and
halving in the context of \emph{topological dynamical systems},
by which we mean pairs~$(X,T)$,
where~$X$ is a compact metric space and~$T:X\to X$
is a homeomorphism
with a fixed point
and with finitely many points of period~$n$
for each~$n\ge1$.
Halving (and doubling) may be thought of as
a relationship between pairs~$(X,T)$ and~$(\overline{X},\overline{T})$
of the following form. 

Suppose there is
a continuous involution~$\imath:X\to X$
that commutes with~$T$, and use
this to define an equivalence relation on~$X$
by saying that~$x\sim y$ if and only if~$x=\imath(y)$.
Writing~$[x]_{\mathord\sim}$ for the
equivalence class~$\{x,\imath(x)\}$ of~$x$
we define~$\overline{X}$ to be the quotient
space~$X/\mathord\sim$ and~$\overline{T}$ to be
the map defined by~$\overline{T}\([x]_{\sim}\)=[T(x)]_{\sim}$.
The process of passing from~$(X,T)$ to~$(\overline{X},\overline{T})$
may be thought of as halving, and passing from~$(\overline{X},\overline{T})$
to~$(X,T)$ as doubling.

The question we address is to ask which structures
are preserved and which are not by doubling and halving
in this sense.
The most important quantity associated to a topological
dynamical system is the topological entropy, 
which will not be of interest here as it is preserved
by halving (see~\cite[Ch.~5]{elw} for the
definition and details).
We are particularly interested
in the relationship between closed orbits in the
two systems. It turns out that there are some constraints on the
relationship between the numbers of closed orbits for~$T$ and for~$\overline{T}$
but that, within these constraints, essentially everything is possible.
The constraints will be described in Lemma~\ref{bounds}, and the
freedom within those constraints in Corollary~\ref{theorem}.
%%%%%%%%%%%%%%%%%%%%%%%%%%%%%%%%%%%%%%%%%%%%%%%%%%%%%%%%
%%\todo[color=orange!40]{Need to refer to some results from later, not sure which yet.}
%%%%%%%%%%%%%%%%%%%%%%%%%%%%%%%%%%%%%%%%%%%%%%%%%%%%%%%%

%%%%%%%%%%%%%%%%%%%%%%%%%%%%%%%%%%%%%%%%%%%%%%%%%%%%%%%%
\section{Closed orbits}
%%%%%%%%%%%%%%%%%%%%%%%%%%%%%%%%%%%%%%%%%%%%%%%%%%%%%%%%

We begin with some notational conventions for a map~$T:X\to X$. For~$n$ a natural number, we write
\[
\fixset_T(n)=\{x\in X\mid T^n(x)=x\}
\]
for the set of points of period~$n$ under
iteration of~$T$,
and~$\fixnumber_T(n)=\vert\fixset_T(n)\vert$
for the number of points of period~$n$.
Similarly, write~$\orbitset_T(n)$ for the
set of closed orbits of length~$n$ under iteration
of~$T$, and~$\orbitnumber_T(n)=\vert\orbitset_T(n)\vert$
for the number of closed orbits of length~$n$.
The set of points of period~$n$ comprises exactly
the disjoint union of those points on a closed
orbit of length~$d$
for each~$d$ dividing~$n$, and each such
orbit consists of~$d$ distinct points.
Thus
\begin{equation}\label{fixpointsfromorbits}
\fixnumber_T(n)=\sum_{d\vert n}d\orbitnumber_T(d),
\end{equation}
and hence, by M{\"o}bius inversion,
\[
\orbitnumber_T(n)=
\textstyle\frac1n
\displaystyle\sum\limits_{d\vert n}
\mu\(\textstyle\frac{n}{d}\)
\fixnumber_T(d)
\]
for any~$n\ge1$.

A convenient generating function for
the periodic point data is the
dynamical zeta function
\[
\zeta_T(z)=\exp\(\sum_{n\ge1}
\fixnumber_T(n)\frac{z^n}{n}\)
=
\prod_{n\ge1}\(1-z^{n}\)^{-\orbitnumber_T(n)},
\]
(the second equality is equivalent to the identity~\eqref{fixpointsfromorbits})
which defines a function under the
assumption that~$\fixnumber_T(n)<\infty$ for all~$n\ge1$,
in which case it has radius
of convergence given by
\[
\(\limsup_{n\to\infty}\(\fixnumber_T(n)\)^{1/n}\)^{-1}.
\]

%%%%%%%%%%%%%%%%%%%%%%%%%%%%%%%%%%%%%%%%%%%%%%%%%%%%%%%%
\section{Closed orbits and topological factors}
%%%%%%%%%%%%%%%%%%%%%%%%%%%%%%%%%%%%%%%%%%%%%%%%%%%%%%%%

The halving process of the introduction is
a special case of forming a
\emph{topological factor}. If~$(X,T)$ and~$(Y,S)$ are
topological
dynamical systems with a continuous surjective
map~$\pi:X\to Y$
satisfying~$\pi\circ T=S\circ\pi$,
then~$(Y,S)$ is called a topological
factor of~$(X,T)$. In general, closed orbits can behave very badly under a
topological factor map, as the following examples illustrate.

\begin{example}\label{factorwithfewclosedorbits}
Closed orbits in~$(X,T)$ may be glued together
in a topological factor, as illustrated in
Figure~\ref{figurefactor}(a). An extreme instance
of this is to take~$Y=\{y\}$ to be a singleton,
set~$S(y)=y$,
and define the factor map~$\pi$ by~$\pi(x)=y$
for all~$x\in X$. Then, whatever the
sequence~$\(\orbitnumber_T(n)\)$ of numbers
of closed orbits under~$T$, the factor
system has~$\orbitnumber_S(1)=1$ and~$\orbitnumber_S(n)
=0$ for all~$n\ge2$.
\end{example}

\begin{example}\label{factorwithmanyclosedorbits}
Finite pieces of orbits in~$(X,T)$ that are
not closed may close up under the factor map,
producing closed orbits in~$(Y,S)$, as illustrated
in Figure~\ref{figurefactor}(b). An extreme instance
of this is once again to take any topological
dynamical system~$(Y,S)$ and form the
product system~$X=Y\times\mathbb T$, where~$\mathbb T=\mathbb R/\mathbb Z$ is the additive circle, with
map~$T(y,t)=(S(y),t+\alpha\pmod{1})$ for
some fixed~$\alpha\notin\mathbb Q$. Then~$(Y,S)$ is a topological factor of~$(X,T)$ via the map
\begin{eqnarray*}
\pi:X&\longrightarrow& Y\\
(y,t)&\longmapsto& y.
\end{eqnarray*}
Then, whatever the
sequence~$\(\orbitnumber_S(n)\)$ of numbers
of closed orbits under~$S$, the map~$T$ has no orbits of finite length.
\end{example}

\begin{figure}[htbp]\begin{center}
\psfrag{A}{$\pi$}\psfrag{B}{$\pi$}
\psfrag{a}{(a)}\psfrag{b}{(b)}
\scalebox{1}[1]{\includegraphics{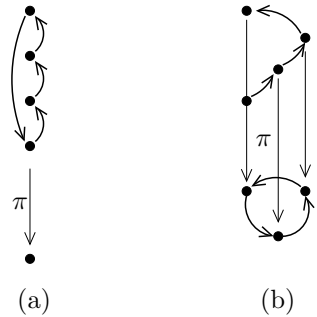}}
\caption{\label{figurefactor} Closed orbits squashed to a point and
created under a factor map.}
\end{center}
\end{figure}

Thus we cannot expect to be able to make
statements about numbers of closed orbits in
factor systems in general.
For quotienting by an action of~$C_2$ (that is, halving),
the situation is more restricted.

\begin{example}\label{circledoublingtotentmap}
Let~$X=\mathbb T=\mathbb R/\mathbb Z$ be the additive
circle, and define the map~$T$ on~$X$ by~$T(x)=2x \pmod{1}$.
This is not a homeomorphism, but is
a convenient familiar map to use as an
initial example.
The involution~$\imath(x)=1-x$ commutes with~$T$,
and so defines a halved system~$(\overline{X},\overline{T})$.
A convenient way to visualize this system is to imagine
looking
sideways at the unit circle so that points identified
by the map~$\imath$ are seen as a single point,
as illustrated in Figure~\ref{figuretentmap}.
\begin{figure}[htbp]\begin{center}
\psfrag{1}{$0$}\psfrag{m1}{$\textstyle\frac12$}
\psfrag{i}{$\textstyle\frac34$}\psfrag{-i}{$\textstyle\frac14$}
\psfrag{A}{$[0]=\{0\}$}\psfrag{B}{$[\textstyle\frac12]=\{\textstyle\frac12\}$}
\psfrag{C}{$[0]$}\psfrag{D}{$[\textstyle\frac12]=\{\textstyle\frac12\}$}
\psfrag{II}{$[\textstyle\frac14]=\{\textstyle\frac14,\textstyle\frac34\}$}
\scalebox{1}[1]{\includegraphics{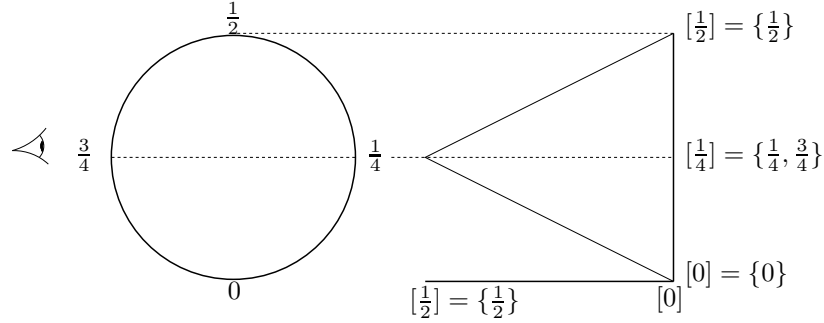}}
\caption{\label{figuretentmap} Halving
the circle doubling map gives a tent map.}
\end{center}
\end{figure}
It is straightforward to check that the quotient
system~$(\overline{X},\overline{T})$ is the tent map,
with~$\overline{X}=[0,\frac12]$
and
\[
\overline{T}(x)=\begin{cases}2x&0\le x\le\frac14,\\
1-2x&\frac14\le x\le\frac{1}{2}.\end{cases}
\]
For these two systems it is easy
to calculate that
\[
\zeta_{T}(z)=\frac{1-z}{1-2z}
\]
and
\[
\zeta_{\overline{T}}(z)=\frac{1}{1-2z}.
\]
In this example, a system with approximately~$2^n$
points of period~$n$ and a rational zeta function
is halved to a system with the same properties.
\end{example}

The following example gives a natural way in which one
could double a topological dynamical system,
simply by putting together two copies of
the system.

\begin{example}\label{double}
Given any dynamical system~$(Y,S)$
on a metric space~$(Y,d)$
we may form the doubled space~$X=Y\times\{0,1\}$ with the
metric
\[
d_X\((y,e),(y',e')\)
=d(y,y')+\begin{cases}
1&\mbox{if }e\neq e',\\
0&\mbox{if }e=e',
\end{cases}
\]
for~$y,y'\in Y$ and~$e,e'\in\{0,1\}$,
%%%%%%%%%%%%%%%%%%%%%%%%%%%%%%%%%%%%%%%%%%%%%%%%%%%%%%%%
%%\todo[color=orange!40]{Changed to~$y$s.}
%%%%%%%%%%%%%%%%%%%%%%%%%%%%%%%%%%%%%%%%%%%%%%%%%%%%%%%%
and define a map on the doubled space by
\begin{eqnarray*}
T:X&\longrightarrow& X\\
(y,e)&\longmapsto& (S(y),e+1\pmod{2}).
\end{eqnarray*}
The involution~$\imath:(y,e)\longmapsto(y,e+1\pmod{2})$
commutes with~$T$, giving a halved
system~$(\overline{X},\overline{T})$ which
can be identified with the original system~$(Y,S)$.
Clearly
\begin{equation}
\fixnumber_T(n)=\begin{cases}
0&\mbox{if $n$ is odd;}\\
2\fixnumber_S(n)&\mbox{if $n$ is even,}
\end{cases}\label{doubled}
\end{equation}
so that
\[
\limsup_{n\to\infty} \tfrac 1n \log\fixnumber_T(n) = \limsup_{n\to\infty} \tfrac 1n \log\fixnumber_S(n),
\]
meaning that
%%%%%%%%%%%%%%%%%%%%%%%%%%%%%%%%%%%%%%%%%%%%%%%%%%%%%%%%
%%\todo[color=orange!40]{Changed to ``same radius of convergence''.}
%%%%%%%%%%%%%%%%%%%%%%%%%%%%%%%%%%%%%%%%%%%%%%%%%%%%%%%%
the zeta functions~$\zeta_T$ and~$\zeta_{\overline T}=\zeta_S$ have the same radius of convergence.
% the~$\limsup$ logarithmic growth rate is unchanged in the quotient.
Moreover, the relation~\eqref{doubled} may be written as
\begin{equation}\label{coolquestion}
\zeta_T(s)=\zeta_S(s)\zeta_S(-s),
\end{equation}
showing that if~$\zeta_{S}$ is rational then~$\zeta_T$ is also rational,
though we will see later that the converse is not true (see Example~\ref{doublereprise}).
\end{example}

Our purpose here is to show how extremely unrepresentative
Examples~\ref{circledoublingtotentmap} and~\ref{double} really are.
In general, both the growth rate in the number of periodic points and the arithmetic
nature of the zeta function do not survive under doubling or halving.
However, in contrast to Examples~\ref{factorwithfewclosedorbits}
and~\ref{factorwithmanyclosedorbits} we
will show that the change of the growth rate in the number of closed orbits is restricted.

%%%%%%%%%%%%%%%%%%%%%%%%%%%%%%%%%%%%%%%%%%%%%%%%%%%%%%%%
\section{Shortening, Surviving, and Gluing Orbits}
%%%%%%%%%%%%%%%%%%%%%%%%%%%%%%%%%%%%%%%%%%%%%%%%%%%%%%%%

We put ourselves back in the halving situation of the introduction. Thus~$(X,T)$ is a topological dynamical system, which we recall means a pair with~$X$ a compact metric space and~$T:X\to X$ a homeomorphism with a fixed point and with finitely many points of period~$n$ for each~$n\ge1$, and~$\imath$ is a continuous involution on~$X$ which commutes with~$T$. Then~$\overline{X}$ is the quotient of~$X$ under the equivalence relation induced by~$\imath$, with quotient map~$\pi:X\mapsto\overline{X}$ given by~$\pi(x)=\{x,\imath(x)\}$. Note that~$\overline{X}$ is a metric space: if~$d_X$ is the metric on~$X$, then the metric~$d_{\overline{X}}$ on~$\overline{X}$ is given by
\[
d_{\overline{X}}(\overline{x},\overline{y})=\min\{d_X(x,y)\mid x\in\pi^{-1}(\overline{x}),y\in\pi^{-1}(\overline{y})\}.
\]
The map~$\overline{T}:\overline{X}\to\overline{X}$ is (well-)defined by the relation~$\pi\circ T=\overline T\circ\pi$, and is a homeomorphism.

The factor map~$\pi$ maps any closed orbit
under~$T$ to a closed orbit under~$\overline{T}$; conversely,
since the fibres of~$\pi$ are finite, the inverse image under~$\pi$ of a closed orbit under~$\overline{T}$ is a finite set closed under~$T$ so is a finite union of closed orbits under~$T$. In particular,~$\pi$ induces a surjective map
\begin{equation}\label{eqn:pionto}
\bigsqcup_{n=1}^\infty \orbitset_T(n) \to \bigsqcup_{n=1}^\infty \orbitset_{\overline{T}}(n).
\end{equation}

In order to analyze this more closely, let~$\tau=\{x,T(x),T^2(x),\dots,T^n(x)=x\}$ be a closed
orbit in~$(X,T)$ of length~$n$.
%%%%%%%%%%%%%%%%%%%%%%%%%%%%%%%%%%%%%%%%%%%%%%%%%%%%%%%%
%%\todo[color=orange!40]{Changed from~$k$ to~$n$ since we use~$n$ below.}
%%%%%%%%%%%%%%%%%%%%%%%%%%%%%%%%%%%%%%%%%%%%%%%%%%%%%%%%
Then~$\imath$ and~$\tau$ can interact
in just three ways.
\begin{enumerate}
\item It could fix the orbit~$\tau$ pointwise,
(we say that~$\tau$ `survives'),
that is~$x=\imath(x)$ for all~$x\in\tau$; we
write~$\orbitset_T^s(n)$ for the set of closed orbits of length~$n$
under~$T$ that are fixed pointwise by~$\imath$. Then the factor map~$\pi$ induces an injective map~$\orbitset_T^s(n)\to\orbitset_{\overline T}(n)$.
\item It could map~$\tau$ to another closed orbit~$\tau'$ of the
same length (we say that~$\tau$ is `glued' to~$\tau'$), that is~$\imath(x)\in\tau'$ for all~$x\in\tau$;
we write~$\orbitset_T^{g}(n)$ for the set of closed
orbits of length~$n$ under~$T$ that are glued together in pairs. The factor map~$\pi$ induces a~$2$-to-$1$ map~$\orbitset_T^g(n)\to\orbitset_{\overline T}(n)$.
\item It could preserve the orbit~$\tau$ but not fix it pointwise; in that case, we must
have~$n=2k$ even and~$\imath(x)=T^{k}(x)$,
%%%%%%%%%%%%%%%%%%%%%%%%%%%%%%%%%%%%%%%%%%%%%%%%%%%%%%%%
%%\todo[color=orange!40]{Reverted to~$n=2k$ to be consistent with length~$n$ orbits for~$T$ in the other cases.}
%%%%%%%%%%%%%%%%%%%%%%%%%%%%%%%%%%%%%%%%%%%%%%%%%%%%%%%%
for all~$x\in\tau$, so that the orbit~$\pi(\tau)$ of~$\overline T$ has length~$k$ (we say that~$\tau$ is `halved').
we write~$\orbitset_T^h(n)$
for the set of closed orbits of
length~$n$ under~$T$ that are halved in
length, and~$\pi$ induces an injective map~$\orbitset_T^h(2k)\to\orbitset_{\overline T}(k)$.
\end{enumerate}
Clearly
\[
\orbitset_T(n)=\orbitset_T^h(n)\sqcup\orbitset_T^g(n)\sqcup\orbitset_T^s(n)
\]
is a disjoint union, and so
\begin{equation}\label{splitsumintothree}
\orbitnumber_T(n)=\orbitnumber_T^h(n)+\orbitnumber_T^g(n)+\orbitnumber_T^s(n),
\end{equation}
for any~$n\ge1$. Similarly, from the surjective map~\eqref{eqn:pionto} and the three possible behaviours, we get
\[
\orbitset_{\overline T}(n)=\pi\(\orbitset_T^h(2n)\)\sqcup\pi\(\orbitset_T^g(n)\)\sqcup\pi\(\orbitset_T^s(n)\),
\]
and it follows that
\begin{equation}\label{splitsumintothreeforquotientsystem}
\orbitnumber_{\overline{T}}(n)=\orbitnumber_T^h(2n)+\textstyle\frac12\orbitnumber_T^g(n)+\orbitnumber_T^s(n),
\end{equation}
for any~$n\ge1$. Since these numbers are finite and
\[
\orbitnumber_{\overline{T}}(1)
\ge
\orbitnumber_T(1)-\textstyle\frac 12\orbitnumber_T^g(1)
\ge
\textstyle\frac 12\orbitnumber_T(1)>0,
\]
we see that~$(\overline{X},\overline{T})$ is also a topological dynamical system.

The way in which the set of orbits of length~$n$ under~$T$
decomposes into those that halve in length, those that glue together, and
those that survive, is not arbitrary. Whatever constraints on~$\imath$
arise from having to be a continuous involution on~$X$ that
commutes with~$T$, there are some purely combinatorial
constraints as follows:
\begin{equation}\label{constraint1}
\orbitnumber_T^h(n)=0 \mbox{ if~$n$ is odd};
\end{equation}
\begin{equation}\label{constraint2}
\orbitnumber_T^g(n) \mbox{ is even for all~$n\ge1$}.
\end{equation}
These observations already constrain
the effect that halving can have on the growth
rate of closed orbits.

\begin{lemma}\label{bounds}
Suppose~$(\overline{X},\overline{T})$ is obtained
from~$(X,T)$ by halving. Then, for any~$n\ge1$,
\begin{enumerate}
\item[{\rm(a)}]
$\frac 12 \fixnumber_T(n)
\le
\fixnumber_{\overline{T}}(n)
\le
\frac 12 \(\fixnumber_T(n) +\fixnumber_T(2n)\)$;
\item[{\rm(b)}]
$\orbitnumber_{\overline{T}}(n)
\le
\orbitnumber_T(n)+ \orbitnumber_T(2n)$, and if~$n$ is odd then~$\orbitnumber_{\overline{T}}(n)\ge \frac 12 \orbitnumber_T(n)$.
%%%%%%%%%%%%%%%%%%%%%%%%%%%%%%%%%%%%%%%%%%%%%%%%%%%%%%%%
%%\todo[color=orange!40]{Left out ``in addition'' to keep it on one line.}
%%%%%%%%%%%%%%%%%%%%%%%%%%%%%%%%%%%%%%%%%%%%%%%%%%%%%%%%
\end{enumerate}
\end{lemma}

\begin{proof}
(a) The lower bound for~$\fixnumber_{\overline{T}}(n)$ comes from the fact that the fibres of the factor map~$\pi$ have cardinality at most~$2$; since~$\pi$ maps~$\fixset_T(n)$ to~$\fixset_{\overline{T}}(n)$, we
deduce that~$\fixnumber_{\overline{T}}(n)\ge\frac 12\fixnumber_T(n)$.
The bound is achieved if all orbits of length dividing~$n$ glue together in pairs.

The upper bound comes from the containment~$\pi^{-1}(\fixset_{\overline{T}}(n))\subseteq\fixset_T(2n)$. If~$x\in\fixset_T(2n)\setminus\fixset_T(n)$, then~$\pi(x)\in\fixset_{\overline{T}}(n)$ if and only if~$x$ lies in an orbit which halves in length, in which case~$x$ lies in a fibre of cardinality~$2$. Thus
\begin{eqnarray*}
\fixnumber_{\overline{T}}(n) &=& \left\vert\fixset_{\overline{T}}(n)\cap\pi(\fixset_T(2n)\setminus\fixset_T(n))\right\vert + \left\vert\fixset_{\overline{T}}(n)\cap\pi(\fixset_T(n))\right\vert \\
&\le& \tfrac 12\(\fixnumber_T(2n)-\fixnumber_T(n)\) + \fixnumber_T(n) \\
& =& \tfrac 12\(\fixnumber_T(2n)+\fixnumber_T(n)\).
\end{eqnarray*}
The upper bound is achieved if all orbits of order dividing~$n$ survive, while all other orbits of order dividing~$2n$ halve.

(b) For the upper bound for~$\orbitnumber_{\overline{T}}(n)$,
we have
\[
\orbitnumber_{\overline{T}}(n)
=
\orbitnumber^h_T(2n)+\tfrac{1}{2}\orbitnumber^g_T(n)+
\orbitnumber^s_T(n)
\le
\orbitnumber_{T}(2n)+ \orbitnumber_T(n),
\]
by~\eqref{splitsumintothree} and~\eqref{splitsumintothreeforquotientsystem}. Again this is achieved when all orbits of
order~$n$ survive, while all orbits of order~$2n$ halve.

For the lower bound, when~$n$ is odd we
have~$\orbitnumber_T^h(n)=0$ so that
\[
\orbitnumber_T(n)=\orbitnumber_T^s(n)+\orbitnumber_T^g(n)
\]
and, from~\eqref{splitsumintothreeforquotientsystem},
\[
\orbitnumber_{\overline{T}}(n)
\ge \tfrac{1}{2}\orbitnumber^g_T(n)+\orbitnumber^s_T(n)
\ge \tfrac{1}{2}\orbitnumber_T(n).
\]
This bound is achieved if all orbits of length~$n$ are glued in pairs.
\end{proof}

\begin{remark}
Note that Lemma~\ref{bounds} does not give a lower bound for~$\orbitnumber_{\overline{T}}(n)$ when~$n$ is even. Specifically, when~$n$ is even, all orbits of length~$n$ might halve in length
while orbits of length~$2n$ retain their length. Thus the only possible lower bound is the trivial one~$\orbitnumber_{\overline{T}}(n)\ge 0$.
\end{remark}

In the case that~$\orbitnumber_T(n)$ grow exponentially,
Lemma~\ref{bounds} immediately gives bounds on the logarithmic growth rate of~$\orbitnumber_{\overline{T}}(n)$.

\begin{corollary}\label{cor:expbounds}
%%%%%%%%%%%%%%%%%%%%%%%%%%%%%%%%%%%%%%%%%%%%%%%%%%%%%%%%
%%\todo[color=green!40]{I don't think this Corollary needs a proof. Do you agree?}
%%%%%%%%%%%%%%%%%%%%%%%%%%%%%%%%%%%%%%%%%%%%%%%%%%%%%%%%
Let~$(X,T)$ be a topological dynamical system, let~$\imath$ be a continuous involution on~$X$ commuting with~$T$, and let~$(\overline{X},\overline{T})$ be the halved system. Suppose there is a real number~$\lambda>0$ such that~$\limsup\limits_{n\to\infty} \frac 1n\log\fixnumber_T(n)=\lambda$. Then
\[
\lambda\le \limsup_{n\to\infty} \tfrac 1n\log\fixnumber_{\overline T}(n)\le 2\lambda.
\]
\end{corollary}

We will see later (see Corollary~\ref{theorem}) that every growth rate in the closed interval~$[\lambda,2\lambda]$ is obtainable.

%%%%%%%%%%%%%%%%%%%%%%%%%%%%%%%%%%%%%%%%%%%%%%%%%%%%%%%%
\section{The basic lemma}
%%%%%%%%%%%%%%%%%%%%%%%%%%%%%%%%%%%%%%%%%%%%%%%%%%%%%%%%

Our main observation is that, if we are free to choose
the topological dynamical systems, then~\eqref{constraint1} and~\eqref{constraint2} are the
only constraints on the behaviour
of closed orbits under halving.
This is a simple extension of
an elementary remark in~\cite{MR1873399}: for any sequence~$(a_n)$
of non-negative integers, there is a topological dynamical system~$(X,T)$
with~$\orbitnumber_T(n)=a_n$ for all~$n\ge1$.

\begin{lemma}\label{basiclemma}
Let~$(b_n^s)$,~$(b_n^g)$, and~$(b_n^h)$ be three sequences of
non-negative integers, with~$b_1^s>0$. Define sequences~$(a_n^h)$,~$(a_n^g)$, and~$(a_n^s)$ by
\[
a_n^s=b_n^s,\quad
a_n^g=2b_n^g,
\quad\mbox{and}\quad
a_n^h=\begin{cases} b_{n/2}^h &\mbox{if~$n$ is even},\\ 0&\mbox{otherwise.} \end{cases}
\]
Define~$a_n=a_n^s+a_n^g+a_n^h$ and~$b_n=b_n^s+b_n^g+b_n^h$ for all~$n\ge 1$.
Then there are a topological dynamical system~$(X,T)$ and a continuous
involution~$\imath:X\to X$ that commutes with~$T$, such that~$\orbitnumber_T(n)=a_n$
and~$\orbitnumber_{\overline{T}}(n)=b_n$, for all~$n\ge1$.
\end{lemma}

Notice in particular that taking~$b_n^s=a_n$ and~$b_n^g=b_n^h=0$, for all~$n\ge1$
recovers the basic lemma of~\cite{MR1873399}.

Before giving an algebraic proof, we give a more
geometric sketch to give an idea of what is happening. We can construct~$X$ as a closed (and hence
compact) subset of the triangle~$\{(x,y)\mid
0\le x\le1,0\le y\le x\}\subset\mathbb R^2$,
with the metric inherited from~$\mathbb R^2$.
Above each
point~$(\frac1n,0)$ for~$n\ge2$ draw~$a_n$
disjoint regular~$n$-gons in such
a way that all of them are disjoint as subsets
of the plane.
By hypothesis~$a_1\ge1$, and we draw~$a_1-1$
points ($1$-gons) above~$(1,0)$. Finally locate a single~$1$-gon at~$(0,0)$
(which may be thought of as a point `at infinity').
The space~$X$ is now defined to be the
union of all the vertices of the polygons. (See Figure~\ref{triangle}.)
It is closed because all but one point is isolated,
and the accumulation point~$(0,0)$ is,
by construction, a member of~$X$.
Number the vertices of each~$n$-gon with
the numbers~$1$ to~$n$ consecutively clockwise,
so that we may speak of the `same' point on
two disjoint~$n$-gons
as being the point with the same symbol.
The homeomorphism~$T$ is defined to be the map that
takes each point on any~$n$-gon to the next
point in a clockwise orientation around the same~$n$-gon
(equivalently, adding one using the labels;
the action of this map is illustrated by the
lines joining vertices of the polygons
in Figure~\ref{triangle}).
This defines a homeomorphism since all but one
point is isolated in~$X$, and all the points
close to the fixed point~$(0,0)$ are moved
by a very small distance.
Then~$(X,T)$
is a topological dynamical system, and
by construction~$\orbitnumber_T(n)=a_n$ for
all~$n\ge1$.

\begin{figure}[htbp]\begin{center}
\psfrag{1}{$1$}\psfrag{2}{$\textstyle\frac12$}
\psfrag{3}{$\textstyle\frac13$}
\psfrag{4}{$\textstyle\frac14$}
\psfrag{I}{$\infty$}
\scalebox{1}[1]{\includegraphics{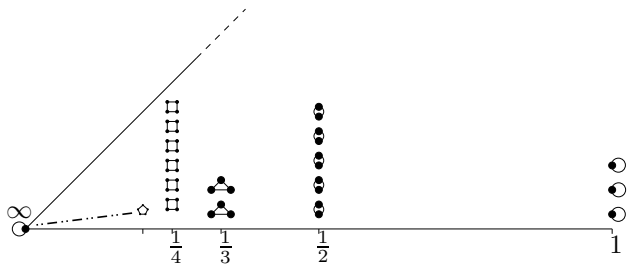}}
\caption{\label{triangle} Building the system $(X,T)$.}
\end{center}
\end{figure}

Now we define an action of~$C_2$ on~$X$ using the
numbers~$a_n^s,a_n^g$, and~$a_n^h$ as follows.
\begin{itemize}
\item For each~$n\ge1$ pick~$\frac12a_n^g$
pairs of~$n$-gons above~$(\frac 1n,0)$ and
define the action of~$\imath$ to send
a point on any one of them to the same point
on the paired~$n$-gon
(these are the glued orbits).
\item For each~$n\ge1$ pick~$a_n^h$ of the~$n$-gons
above~$(\frac 1n,0)$, chosen from those that have not been
chosen already for gluing, and
on each polygon (which will by hypothesis
have even length) define the
action of~$\imath$ to be rotation by~$\pi$ (these are
the halved orbits).
\item On all the remaining points that
are vertices of polygons that are neither
glued nor halved, define~$\imath$ to be
the identity map (these are the
surviving orbits).
\end{itemize}
It is clear that~$\imath$ is continuous (since all points close to the fixed point are moved by a very small distance) and commutes with~$T$
and so defines a halved
system~$(\overline{X},\overline{T})$ which, by construction, has the required numbers of orbits.

Note that, in the proof below, we do not refer to the triangle in the plane and the metric on~$X$ we give is not the same one as in this sketch, but it does give the same topology.

\begin{proof}[Proof of Lemma~\ref{basiclemma}]
We write~$C_2=\{e,\imath\}$, where $e$ is the identity element, for the cyclic group of order~$2$. We begin by describing~$X$ as a set, before compactifying. It will take the form
\[
X=\bigsqcup\limits_{n\geq 1} X_n,
\]
where each~$X_n$ will be the union of closed orbits of length~$n$.
We set
\[
X_n=X_n^s\sqcup X_n^g\sqcup X_n^h,\quad\mbox{where}\quad\begin{cases}
X_n^s = \{1,2,\ldots,a_n^s\} \times \Z/n\Z, \\
X_n^g = \{1,2,\ldots, b^g_n\} \times C_2 \times \Z/n\Z, \\
X_n^h = \{1,2,\ldots,a_n^h\} \times \Z/n\Z.
\end{cases}
\]
We define~$T:X\to X$ and~$\imath:X\to X$ by describing their restrictions to each of the sets~$X_n^s,X_n^g,X_n^h$, which will be preserved:
\begin{itemize}
\item for~$x=(i,k)\in X_n^s$, put~$T(x)=(i,k+1\pmod{n})$ and~$\imath(x)=x$;
\item for~$x=(i,\gamma,k)\in X_n^g$, put~$T(x)=(i,\gamma,k+1\pmod{n})$ and~$\imath(x)=(i,\imath\gamma,k)$;
\item for~$x=(i,k)\in X_n^h$, so that~$n$ is even, we put~$T(x)=(i,k+1\pmod{n})$ and~$\imath(x)=(i,k+\frac n2\pmod{n})$.
\end{itemize}
Then~$\imath$ commutes with~$T$ and, by construction, the map~$T$ and the induced map~$\overline{T}$ on the quotient set~$\overline{X}$ have the required numbers of orbits.

It remains to show that~$X$ can be given the structure of a compact metric space, with respect to which~$T$ and~$\imath$ are homeomorphisms. To do this, we pick a point in~$X_1^s$ (which is non-empty by hypothesis) and call it~$\infty$ and define a metric as follows: if~$x\in X_m$ and~$y\in X_n$, with~$x\not\in{y,\infty}$, then
\[
d(x,y) = d(y,x) = \begin{cases}
\frac{1}{m}& \mbox{if } y=\infty, \\
\frac{1}{\min\{m,n\}}& \mbox{otherwise},
\end{cases}
\]
and~$d(x,x)=0$. It is straightforward to check that this does indeed define a metric and, given any open set~$U$ containing~$\infty$, there exists~$N\ge 1$ such that~$U$ contains~$\bigsqcup_{n\ge N} X_n$ so that~$X\setminus U$ is finite; hence~$X$ is compact. Moreover, since~$T$ and~$\imath$ preserve the sets~$X_n$ and the point~$\infty$, they are isometries, so homeomorphisms, and we are done.
\end{proof}

%%%%%%%%%%%%%%%%%%%%%%%%%%%%%%%%%%%%%%%%%%%%%%%%%%%%%%%%
\section{Growth in closed orbits}
%%\todo[color=red!40]{Not a good section title}}
%%%%%%%%%%%%%%%%%%%%%%%%%%%%%%%%%%%%%%%%%%%%%%%%%%%%%%%%

Lemma~\ref{basiclemma} shows that any pair of sequence~$(a_n),(b_n)$ for which the combinatorial constraints~\eqref{splitsumintothree}--\eqref{constraint2} are satisfied does in fact arise as the orbit count of a pair of systems related by halving. However, it is not so easy to give conditions directly on the sequences~$(a_n),(b_n)$ which guarantee that the combinatorial constraints are indeed satisfied. The following result gives some sufficient conditions.

\begin{proposition}\label{prop:existence}
Let~$(a_n)$ be a sequence of non-negative integers with~$a_1\ge 1$ such that there is an integer~$N\ge 1$ with~$a_{2n}\ge \frac 12a_n$, for all~$n\ge N$. Let~$(b_n)$ be any sequence of integers such that~$b_1>\frac 12a_1$ and
\[
\begin{cases}
\tfrac 12 a_n\le b_n\le a_n &\hbox{for }n<N,\\
\tfrac 12 a_n\le b_n\le a_{2n} &\hbox{for }n\ge N.
\end{cases}
\]
Then there are a topological dynamical system~$(X,T)$ and a continuous
involution~$\imath:X\to X$ that commutes with~$T$, such that~$\orbitnumber_T(n)=a_n$
and~$\orbitnumber_{\overline{T}}(n)=b_n$, for all~$n\ge1$.
\end{proposition}

The conditions on the pair of sequences~$(a_n),(b_n)$ here are not necessary for the existence of a suitable halving system, but they are sufficient for our interests and are not so far from being necessary: for~$n>N$ odd, the condition
\[
\tfrac 12 a_n\le b_n\le a_{2n}
\]
is necessary by Lemma~\ref{bounds}.

\begin{proof} In order to use Lemma~\ref{basiclemma}, we recursively define non-negative integers~$b_n^s,b_n^g$ and~$b_n^h$ such that
\[
b_n=b_n^s+b_n^g+b_n^h,\quad
a_n=b_n^s+2b_n^g+b_{n/2}^h,\quad\mbox{and}\quad
b_n^h\le a_{2n},
\]
where we understand~$b_{n/2}^h=0$ when~$n/2\not\in\Z$. So suppose~$k\ge 1$ and we have defined these for~$n<k$. Then there are two cases.

If~$b_k\le a_k-b_{k/2}^h$, which we note is always the case for~$k<N$, then we put
\[
b_k^g=a_k-b_k-b_{k/2}^h,\quad
b_k^s=b_k-b_k^g,\quad
b_k^h=0.
\]
On the other hand, if~$b_k> a_k-b_{k/2}^h$, then we put
\[
b_k^g=0,\quad
b_k^s=a_k-b_{k/2}^h,\quad
b_k^h=b_k-b_k^s.
\]
These are non-negative and, in the latter case, we have~$b_k^h\le b_k\le a_{2k}$, since~$k\ge N$. Note also that, in either case,~$b_1^s>0$. Now Lemma~\ref{basiclemma} implies the result.
\end{proof}

As a consequence, we get the following result when we consider exponential orbit growth rates.

\begin{corollary}\label{theorem}
Let~$\lambda,\eta,c$ be positive real numbers with~$\lambda>1$ and
%\begin{itemize}
%\item[{\rm(i)}] $\eta=\lambda$ and~$c\geq\frac{1}{2}$, or
%\item[{\rm(ii)}] $\eta\in(\lambda, \lambda^2)$, or
%\item[{\rm(iii)}] $\eta=\lambda^2$ and~$0<c\leq 1$.
%\end{itemize}
\[
\begin{cases} \mbox{$\eta=\lambda$ and~$c\geq\frac{1}{2}$, or}\\
\mbox{$\eta\in(\lambda, \lambda^2)$, or}\\
\mbox{$\eta=\lambda^2$ and~$0<c\leq 1$.}
\end{cases}
\]
Then there exist a topological dynamical system~$(X,T)$
and an involution~$\imath$ on~$X$ commuting with~$T$,
%%%%%%%%%%%%%%%%%%%%%%%%%%%%%%%%%%%%%%%%%%%%%%%%%%%%%%%%
%%\todo[color=orange!40]{We don't need to specify~$X\to X$ since~$(X,T)$ is a topological dynamical system.}
%%%%%%%%%%%%%%%%%%%%%%%%%%%%%%%%%%%%%%%%%%%%%%%%%%%%%%%%
such that
\[
\orbitnumber_T(n)\sim\lambda^n \quad \mbox{and}\quad
\orbitnumber_{\overline{T}}(n)\sim c\eta^n\quad \mbox{as $ n\to\infty $}.
\]
\end{corollary}

\begin{proof}
Let~$N>1$ be any integer such that~$c\eta^N<\lambda^{2N}$ and define sequences by
\[
a_n=\lceil\lambda^n\rceil,\quad
b_n=\begin{cases} a_n&\mbox{if }n<N, \\ \lceil c\eta^n\rceil&\mbox{if }n\ge N,\end{cases}
\]
for~$n\ge 1$. This gives a pair of sequences satisfying the hypotheses of Proposition~\ref{prop:existence}, from which the result follows.
\end{proof}

%%%%%%%%%%%%%%%%%%%%%%%%%%%%%%%%%%%%%%%%%%%%%%%%%%%%%%%%
\section{Dynamical zeta functions}
%%%%%%%%%%%%%%%%%%%%%%%%%%%%%%%%%%%%%%%%%%%%%%%%%%%%%%%%

Bowen and Lanford~\cite[Th.~2]{MR0271401}
showed that there are only countably many
rational dynamical zeta functions, so Corollary~\ref{theorem}
shows in particular that halving and doubling cannot preserve the
property of having a rational zeta function.
We now discuss some examples that give concrete
instances of this phenomenon. The arguments all rely on the
following facts: a power series with positive radius of convergence
represents a rational function if and only if the coefficients
satisfy a linear recurrence (see~\cite[Sec.1.1]{MR1990179}). 
Moreover, the Skolem--Mahler--Lech Theorem
says that, in any linear recurrence sequence~$(a_n)$, the set of
zeros, comprising those values of~$n\in\mathbb N$ for
which~$a_n=0$,
%%%%%%%%%%%%%%%%%%%%%%%%%%%%%%%%%%%%%%%%%%%%%%%%%%%%%%%%
%%\todo[color=red!40]{Is it OK to use~$\mid$ here? We don't want it to be confused with divisibility.}
%%%%%%%%%%%%%%%%%%%%%%%%%%%%%%%%%%%%%%%%%%%%%%%%%%%%%%%%
is the union of a finite set of arithmetic progressions
and a finite set (see~\cite[Ch.~1]{MR1990179} for further details
and a proof).

\begin{example}\label{doublereprise}
We revisit Example~\ref{double}, so that~$(Y,S)$ is
a topological dynamical system and~$X=Y\times\{0,1\}$ with the
map~$T(y,e)=(S(y),e+1\pmod{2})$.
The involution~$\imath:(y,e)=(y,e+1\pmod{2})$
commutes with~$T$, giving the halved
system~$(\overline{X},\overline{T})=(Y,S)$.
There is sufficient freedom in the choice
of orbits of odd length under~$S$ to allow us to find examples
with~$\zeta_{\overline T}=\zeta_S$ irrational but~$\zeta_T$ rational. In particular, we may take
%%%%%%%%%%%%%%%%%%%%%%%%%%%%%%%%%%%%%%%%%%%%%%%%%%%%%%%%
%%\todo[color=yellow!40]{I've had to add~$1$ to each periodic point number, since we need to have~$\fixnumber_S(1)>0$.}
%%%%%%%%%%%%%%%%%%%%%%%%%%%%%%%%%%%%%%%%%%%%%%%%%%%%%%%%
\[
\fixnumber_S(n)=\begin{cases}
2^n+1&\mbox{if $n$ is even;} \\
\sum\limits_{d\vert n} d2^{(d-1)/2}&\mbox{if $ n $ is odd.}
\end{cases}
\]
It is a pleasant exercise to verify that these
really do arise from a topological dynamical system
(that is,~$\frac1n\orbitnumber_S(n)$ is a non-negative integer for each~$n$). Then by~\eqref{doubled}
\[
\zeta_T(z)= \frac{1}{(1-z^2)(1-4z^2)}
\]
is rational. On the other hand
\[
z\frac{\zeta_S'(z)}{\zeta_S(z)}
= \sum_{n=1}^\infty \fixnumber_S(n)z^n =
\frac z{1-z}+\frac{4z^2}{1-4z^2}+ \frac{6z^3-4z^5}{(1-2z^2)^2}+ \varphi(z),
\]
where
\[
\varphi(z)
=
\sum_{n=1}^\infty z^{2n+1}\sum_{\substack{ d\vert 2n+1\\ d\neq 1, 2n+1}} d2^{(d-1)/2}.
\]
We claim that~$\varphi$ is irrational, so that~$\zeta_S$ is irrational.
To see this, note that the
coefficient of~$z^n$ in~$\varphi(z)$ vanishes
precisely when~$n$ is even or~$n$ is an odd prime; since the set of primes is infinite, while any arithmetic progression contains composites, the Skolem--Mahler--Lech Theorem implies
the sequence of coefficients cannot be a linear
recurrence sequence and hence~$\varphi$ cannot be
a rational function.
\end{example}

Our next examples use the sum of divisors function~$\sigma(n)=\sum_{d\vert n}d$.
There are sophisticated
bounds for size of~$\sigma(n)$, but for our purposes
it is sufficient to note the trivial bounds
\[
n\le\sigma(n)\le n^2.
\]
It follows that
the complex power series
\[
\theta(z)=\exp\sum_{n\ge1}\sigma(n)\frac{z^n}{n}
\]
has radius of convergence~$1$.
It is known that
\[
\frac{1}{\theta(z)}=
1-z-z^2+z^5+z^7-\cdots,
\]
where the powers of~$z$ are those
of the form~$(3k^2\pm k)/2$
(see, for example, P{\'o}lya and Szeg{\"o}~\cite[Sec.~VIII, Ex.~75]{MR0170986}).
This means that~$\frac{1}{\theta(z)}$ is a power
series with arbitrarily long consecutive sequences
of zero coefficients. Thus, by the Skolem--Mahler--Lech Theorem, the coefficients of~$\frac{1}{\theta(z)}$ are not a linear recurrence sequence and we deduce that~$\theta(z)$ is not a rational function of~$z$.

\begin{example}\label{irrationalwithrationalquotient}
In order to use the irrationality of~$\theta(z)$, we define~$b_n^g=1$ and~$b_n^h=0$, for all~$n\ge 1$, and we choose~$b_n^s$ later. Now we define~$b_n,a_n$ as in Lemma~\ref{basiclemma} and denote by~$(X,T)$,~$(\overline{X},\overline{T})$ the pair of systems given there. Thus~$T$ has one extra orbit in each length, compared to~$T$, and the action of~$\imath$ on~$X$ has the effect of gluing together
exactly one pair of orbits of each length.

Now we first take~$b_n^s=\frac1n\sum_{d\vert n}\mu\(\frac{n}{d}\)2^d-1$, for~$n\ge 1$, so that~$b_n$ is the number of orbits of length~$n$ in any system with~$2^n$ points of period~$n$. Then
\[
\zeta_{\overline{T}}(z)=\frac{1}{1-2z},
\]
while
\[
\fixnumber_T(n)=\sum_{d\vert n}da_d=2^n+\sigma(n),
\]
so that
\[
\zeta_T(z)=\frac{1}{1-2z}\theta(z).
\]
By the remarks above, this is not a rational function.

In the reverse direction, we take~$b_1^s=1$ and~$b_n^s=\frac1n\sum_{d\vert n}\mu\(\frac{n}{d}\)2^d-2$, for~$n\ge 2$, so that~$a_n$ is the number of orbits of length~$n$ in any system with~$2^n+1$ points of period~$n$. Then
\[
\zeta_{T}(z)=\frac{1}{(1-z)(1-2z)},
\]
while
\[
\fixnumber_{\overline{T}}(n)=\sum_{d\vert n}db_d=2^n+1-\sigma(n),
\]
so that
\[
\zeta_{\overline{T}}(z)=\frac{1}{(1-z)(1-2z)\theta(z)},
\]
which is again irrational.
\end{example}

In fact, the zeta function in the previous example is
worse than irrational.
The function~$1/\theta(z)$ has integer coefficients and radius of
convergence~$1$, but is not rational. Thus,
by the P{\'o}lya--Carlson Theorem (see~\cite{carlson, polya})
%%%%%%%%%%%%%%%%%%%%%%%%%%%%%%%%%%%%%%%%%%%%%%%%%%%%%%%%
%%\todo[color=red!40]{Reference required.}
%%%%%%%%%%%%%%%%%%%%%%%%%%%%%%%%%%%%%%%%%%%%%%%%%%%%%%%%
it has the unit circle as natural boundary, and the function~$\zeta_{\overline{T}}$ also has natural boundary here. Since the radius of convergence of~$\zeta_{\overline{T}}$ is only~$\frac 12$ this is perhaps not so interesting, but our final example shows that it is possible for the circle of convergence and the natural boundary of~$\zeta_{\overline{T}}$ to coincide, even when~$\zeta_T$ is rational.
%%%%%%%%%%%%%%%%%%%%%%%%%%%%%%%%%%%%%%%%%%%%%%%%%%%%%%%%
%%\todo[color=green!40]{Is it worth saying anywhere something about why we can't have a natural boundary on the unit circle?}
%%%%%%%%%%%%%%%%%%%%%%%%%%%%%%%%%%%%%%%%%%%%%%%%%%%%%%%%

\begin{example}\label{naturalboundary}
We begin by recursively defining an auxiliary sequence~$(c_n)$ of non-negative integers by the following conditions:
\begin{itemize}
\item if~$n=1$ or~$n$ is prime, then~$c_n=0$;
\item if~$n$ is composite then, for any prime~$p$ dividing~$n$,
\[
c_n \equiv c_{n/p} \pmod{p^{\ord_p(n)}},
\]
and~$n\le c_n< 2n$.
\end{itemize}
Note that, by the Chinese Remainder Theorem, these conditions determine~$(c_n)$ uniquely.
Now set
%%%%%%%%%%%%%%%%%%%%%%%%%%%%%%%%%%%%%%%%%%%%%%%%%%%%%%%%
%%\todo[color=yellow!40]{I put~$2^{d+3}$ in the definition of~$a_n$ to try to simplify inequalities later. If anyone has a neater way, that would be good!}
%%%%%%%%%%%%%%%%%%%%%%%%%%%%%%%%%%%%%%%%%%%%%%%%%%%%%%%%
\[
a_n=\tfrac 1n \sum_{d\vert n}\mu\(\tfrac nd\)2^{d},\quad
b_n=a_n+\tfrac 1n \sum_{d\vert n}\mu\(\tfrac nd\)c_d2^d.
\]
We assume for now that the sequences~$(a_n),(b_n)$ are non-negative integers satisfying the conditions of Proposition~\ref{prop:existence} and denote by~$(X,T),(\overline{X},\overline{T})$ the systems given there with these numbers of orbits. Then
\[
\zeta_T(z)=\frac 1{1-2z},
\]
while
\[
z\frac{\zeta_{\overline{T}}'(z)}{\zeta_{\overline{T}}(z)}
= \frac{2z}{1-2z} + \sum_{n=1}^\infty c_n 2^nz^n.
\]
Now the integer sequence~$(c_n)$ is not a linear recurrence sequence, since it is zero for all primes and non-zero for all composites, while the bound~$c_n<2n$ implies that the power series
\[
\sum_{n\ge 1}c_nz^n
\]
has radius of convergence~$1$. Hence it has a natural boundary on the unit circle, and we deduce that~$z\frac{\zeta_{\overline{T}}'(z)}{\zeta_{\overline{T}}(z)}$ has a natural boundary on the circle~$|z|=\frac 12$. Thus~$\zeta_{\overline{T}}$ also has a natural boundary here, since it has radius of convergence~$\frac 12$.

It remains to show that~$(a_n),(b_n)$ are non-negative integers and satisfy the conditions of Proposition~\ref{prop:existence}. First, note that~$a_n$ is the number of closed orbits of period~$n$ of the tent map so that~$(a_n)$ is a sequence of non-negative integers. Now we prove that the sequence~$(b_n-a_n)$ is also a sequence of non-negative integers. First we must check that
\[
\sum_{d\vert n}\mu\(\tfrac nd\)c_d2^d
\]
is a non-negative integer divisible by~$n$, for all~$n\ge 1$. To show that
it is divisible by~$n$, we show that it is divisible by~$p^{\ord_p(n)}$, for
each prime~$p$ dividing~$n$. For this, we use the following version of Euler's
generalization of Fermat's Little Theorem: for any prime~$p$ and
integers~$r,m$, with~$p\vert m$, we have
\[
r^{m} \equiv r^{m/p} \pmod{p^{\ord_p(m)}}.
\]
Thus, for any prime~$p$ dividing~$n$,
\[
\sum_{d\vert n}\mu\(\tfrac nd\)c_d2^d =
\sum\limits_{\substack{d\vert n \\ p\nvert \frac nd}}\mu\(\tfrac nd\)\(c_d2^d-c_{d/p}2^{d/p}\)
\equiv 0 \pmod{p^{\ord_p(n)}},
\]
since, for any~$d$ dividing~$n$ with~$p\nvert \frac nd$, we
have~$\ord_p(d)=\ord_p(n)$ and~$c_d\equiv c_{d/p} \pmod{p^{\ord_p(d)}}$ by construction.

For non-negativity, when~$n$ is~$1$ or prime we have~$b_n-a_n=0$. On the other hand, for~$n$ composite, the bounds~$c_n\ge n$ and~$c_d<2d\le n$, for~$d$ a divisor of~$n$,
imply that
\[
\sum_{d\vert n}\mu\(\tfrac nd\)c_d2^d > n2^n-\sum_{d\le \frac n2}n2^d
> n(2^n-2^{\frac n2+1}) > 0.
\]
Finally, for~$n$ composite, the same bounds show that
\[
b_n - a_n =\frac 1n\sum_{d\vert n}\mu\(\tfrac nd\)c_d2^d < (2^n+2^{\frac n2+1})<2^{n+1},
\]
and, similarly,
\[
\frac 1n (2^n - 2^{\frac n2+1}) < a_n < \frac 1n (2^n + 2^{\frac n2+1}) < \frac 1n 2^{n+1}.
\]
Thus, for~$n\ge 6$, using that~$2^{2n}> (2n+3)2^{n+1}$, we have
\[
a_{2n} > \frac 1{2n}(2^{2n} - 2^{n+1}) > \(\frac{n+1}{n}\) 2^{n+1} > (b_n-a_n)+a_n = b_n.
\]
Finally, one checks that~$b_4=19<30=a_8$ and, since~$b_n=a_n<a_{2n}$ for~$n$ prime, the conditions of Proposition~\ref{prop:existence} are satisfied with~$N=2$.
\end{example}

%%%%%%%%%%%%%%%%%%%%%%%%%%%%%%%%%%%%%%%%%%%%%%%%%%%%%%%%
\section{Concluding remarks and questions}
%%%%%%%%%%%%%%%%%%%%%%%%%%%%%%%%%%%%%%%%%%%%%%%%%%%%%%%%

\begin{enumerate}
\item The simple observation in~\cite{MR1873399} that
for any sequence~$(a_n)$ of non-negative integers
there is a topological dynamical system~$(X,T)$
with~$\orbitnumber_T(n)=a_n$ for all~$n\ge1$
was extended by Windsor~\cite{MR2422026}, who showed that
the map may be required to be an infinitely differentiable map on the~$2$-torus.
Does Lemma~\ref{basiclemma}
also have a smooth version, in which both maps and the involution
are differentiable maps on a manifold?
\item We have only considered quotients by
an action of the group~$C_2$. The same process
makes sense if~$(X,T)$ supports an action of
some finite group~$G$ commuting with~$T$, and similar
questions arise. In this setting the structure of the
group plays a larger role, and other
complications arise; this is explored in~\cite{stefi}, where generalizations
of Corollaries~\ref{cor:expbounds} and~\ref{theorem} are obtained.
A particularly interesting sample problem is to understand
a version of the relation~\eqref{coolquestion} for other groups
of symmetries.
\item Achieving radius of convergence strictly smaller
than~$1$ in
Example~\ref{naturalboundary} is important because
with radius of convergence~$1$ the rational part of the
P{\'o}lya--Carlson dichotomy is not particularly
interesting: a rational Taylor series with integer
coefficients and radius of convergence~$1$ has the form~$\frac{p(z)}{(1-z^a)^b}$
for some polynomial~$p$ with integer coefficients and integers~$a,b\ge0$. In our
settings, this would correspond to dynamical systems in which the
number of closed orbits of length~$n$ is constant for large~$n$.
\end{enumerate}
%%%%%%%%%%%%%%%%%%%%%%%%%%%%%%%%%%%%%%%%%%%%%%%%%%%%%%%%
%%\todo[color=orange!40]{Anything else to add?}
%%%%%%%%%%%%%%%%%%%%%%%%%%%%%%%%%%%%%%%%%%%%%%%%%%%%%%%%

%%\bibliographystyle{monthly}
%%\bibliography{references}
%%%%%%%%%%%%%%%%%%%%%%%%%%%%%%%%%%%%%%%%%%%%%%%%%%%%%%%%
%%\todo[color=red!40]{Steffi: title?}
%%%%%%%%%%%%%%%%%%%%%%%%%%%%%%%%%%%%%%%%%%%%%%%%%%%%%%%%

\providecommand{\bysame}{\leavevmode\hbox to3em{\hrulefill}\thinspace}
\providecommand{\MR}{\relax\ifhmode\unskip\space\fi MR }
\providecommand{\MRhref}[2]{%
  \href{http://www.ams.org/mathscinet-getitem?mr=#1}{#2}
}
\providecommand{\href}[2]{#2}

\end{document}